\documentclass[12pt]{article}
\usepackage{graphicx}
\usepackage{amssymb, amsthm, amsmath, amsfonts}
\usepackage{mathrsfs,microtype}
\usepackage{epstopdf,color}
\usepackage{tikz}
\usepackage{hyperref}
\usetikzlibrary{arrows,decorations.pathmorphing,decorations.footprints,fadings,calc,trees,mindmap,shadows,decorations.text,patterns,positioning,shapes,matrix,fit,through,decorations.pathreplacing}
\usepackage[margin=1.3in]{geometry}

\theoremstyle{plain}
\newtheorem{theorem}{Theorem}[section]
\newtheorem{lemma}[theorem]{Lemma}
\newtheorem{corollary}[theorem]{Corollary}

\newtheorem{question}[theorem]{Question}

\theoremstyle{definition}

\newtheorem*{convention}{Convention}

\theoremstyle{remark}

\newcommand{\calG}{\mathcal{G}}

\title{Maximizing $H$-colorings of connected graphs with fixed minimum degree}
\author{John Engbers\thanks{john.engbers@marquette.edu; Department of Mathematics, Statistics and Computer Science, Marquette University, Milwaukee, WI 53201.}}
\date{\today }

\begin{document}

\maketitle

\begin{abstract}
For graphs $G$ and $H$, an $H$-coloring of $G$ is a map from the vertices of $G$ to the vertices of $H$ that preserves edge adjacency.  We consider the following extremal enumerative question: for a given $H$, which connected $n$-vertex graph with minimum degree $\delta$ maximizes the number of $H$-colorings?  We show that for non-regular $H$ and sufficiently large $n$, the complete bipartite graph $K_{\delta,n-\delta}$ is the unique maximizer.  As a corollary, for non-regular $H$ and sufficiently large $n$ the graph $K_{k,n-k}$ is the unique $k$-connected graph that maximizes the number of $H$-colorings among all $k$-connected graphs.  Finally, we show that this conclusion does not hold for all regular $H$ by exhibiting a connected $n$-vertex graph with minimum degree $\delta$ which has more $K_{q}$-colorings (for sufficiently large $q$ and $n$) than $K_{\delta,n-\delta}$.
\end{abstract}



\section{Introduction and statement of results}

Given a simple loopless graph $G = (V(G),E(G))$ and a graph $H = (V(H),E(H))$ without multi-edges but possibly with loops, an \emph{$H$-coloring} of $G$ is an adjacency preserving map $f:V(G) \to V(H)$ (i.e., $f$ satisfies $f(v) \sim_H f(w)$ whenever $v \sim_G w$). An $H$-coloring of $G$ is also referred to as a \emph{graph homomorphism} from $G$ to $H$. When $H=H_{\text{ind}}$, the graph consisting of a single edge and a loop on one endvertex, an $H_{\text{ind}}$-coloring of a graph $G$ may be identified with an independent set in $G$, and when $H=K_{q}$, the complete unlooped graph on $q$ vertices, a $K_{q}$-coloring of $G$ may be identified with a proper $q$-coloring of $G$.  We let $\hom(G,H)$ denote the number of $H$-colorings of $G$.

Much recent work has been done on the following extremal enumerative question: given a family of graphs $\calG$ and a graph $H$, which $G \in \calG$ maximizes the quantity $\hom(G,H)$? For a survey of results and conjectures surrounding questions of this type, see \cite{Cutler}.  

A family relevant to the current paper is the family of bipartite $n$-vertex $d$-regular graphs. Here, Galvin and Tetali \cite{GalvinTetali} (extending work of Kahn \cite{Kahn} for independent sets) showed that the graph in this family with the maximum number of $H$-colorings is $\frac{n}{2d} K_{d,d}$, which consists of $\frac{n}{2d}$ disjoint copies of the complete bipartite graph $K_{d,d}$. (Note that throughout we assume $n$ is large and $d$ is fixed, and we will assume any necessary divisibility conditions between $d$ and $n$.) When moving to the larger class of \emph{all} $n$-vertex $d$-regular graphs, for some $H$ the graph $\frac{n}{2d}K_{d,d}$ still maximizes the number of $H$-colorings (see \cite{Zhao,Zhao2}) but for other graphs $H$ the graph $\frac{n}{d+1}K_{d+1}$, which consists of $\frac{n}{d+1}$ disjoint copies of the complete graph $K_{d+1}$, maximizes the number of $H$-colorings (for example, take $H$ to be the disjoint union of two looped vertices).  Recently Sernau \cite{Sernau} showed that for certain $H$ there are disjoint copies of other small graphs (in particular, not copies of $K_{d,d}$ or $K_{d+1}$) that maximize the number of $H$-colorings, and also gave new families of $H$ for which copies of $K_{d+1}$ are the maximizer and other $H$ for which copies of $K_{d,d}$ are the maximizer.  

What drives some of these results is that the maximizing graph $G$ is obtained by taking the disjoint union of some number of copies of the same fixed small graph.  In particular, the idea in \cite{Sernau} to find maximizing graphs other than copies of $K_{d+1}$ or $K_{d,d}$ is to find an $H$ and a graph with the smallest number of vertices (more than $d+1$ but less than $2d$) that has a non-zero number of $H$-colorings, and then to let $H'$ be a large number of copies of $H$.  In this way, the $H'$-coloring count is driven by having the largest number of components (each with a non-zero number of $H$-colorings of each component) in an $n$-vertex $d$-regular graph. It is worth noting that $H'$ can be modified to be connected with the maximizing graph remaining the same \cite{Sernau}. 
While there is a single graph having the most number of $H$-colorings within the family of \emph{bipartite} $n$-vertex $d$-regular graphs for each possible $H$, at this point it is unclear what the smallest list of $n$-vertex $d$-regular graphs is so that for each possible $H$, some graph in the list maximizes the number of $H$-colorings of graphs in the family.  

Another natural and related family that has been studied is the family  of $n$-vertex graphs with minimum degree $\delta$. Here we have complete extremal results for $H_{\text{ind}}$-colorings \cite{CutlerRadcliffe}. The extremal question for all other $H$ was investigated for $\delta = 1$ and $\delta=2$ in \cite{Engbers}.  For larger values of $\delta$, a general answer is still unknown, but as shown in \cite{Engbers} there is a family of $H$ for which $K_{\delta,n-\delta}$ maximizes the number of $H$-colorings for all $\delta$ when $n$ is sufficiently large.  However, it is not difficult to construct $H$ for which some other graph maximizes the number of $H$-colorings 
for all $\delta$. Indeed, as with regular graphs, there are $H'$ for which having a large number of components with a non-zero number of $H$-colorings significantly increases the $H'$-coloring count, and so the connected graph $K_{\delta,n-\delta}$ does not maximize the number of $H'$-colorings. 

Notice that an $H$-coloring of $G$ requires, by definition, a component of $G$ to be mapped to a component of $H$, and so it is natural to restrict our attention to $n$-vertex graphs with minimum degree $\delta$ which are connected.  Doing this produces a much more coherent structure of the maximizing graphs for almost all $H$, even large $H$. To describe these results, we now state the degree convention that we will use.

\begin{convention}
The degree of a vertex $v$ is $d(v) = |\{w: v \sim w\}|$.  In particular, a loop on a vertex adds one to the degree.  When considering $H$-colorings of $G$, we let $\Delta$ denote the maximum degree of a vertex in $H$.
\end{convention}

It was shown in \cite{Sidorenko} (with a short proof given in \cite{CsikvariLin}) that in the family of trees the star $K_{1,n-1}$ has the largest number of $H$-colorings for any $H$.  In \cite{EngbersGalvin}, this tree result is proved for sufficiently large $n$. The proof technique from \cite{EngbersGalvin} generalizes to a proof for the family of $2$-connected graphs on $n$ vertices, where for sufficiently large $n$ the graph $K_{2,n-2}$ maximizes the number of $H$-colorings for any connected non-regular $H$.   In \cite{EngbersGalvin}, the authors ask what happens for $k$-connected graphs where $k>2$.  We answer this question for connected non-regular $H$ and sufficiently large $n$.

\begin{theorem}\label{thm-main}
Fix $\delta \geq 2$.  Let $H$ be connected and non-regular, and let $G$ be an $n$-vertex connected graph with minimum degree at least $\delta$.  Then there exists a constant $c(\delta,H)$ so that for $n \geq c(\delta,H)$ we have 
\[\hom(G,H) \leq \hom(K_{\delta,n-\delta},H),
\]
with equality if and only if $G = K_{\delta,n-\delta}$. 
\end{theorem}

When $\delta=1$, the conclusion is still true (for all $n$) as a corollary to the result for trees \cite{Sidorenko}. Also, notice that a $k$-connected graph must be connected and have minimum degree at least $k$.   Since $K_{k,n-k}$ is $k$-connected, the following is immediate. 

\begin{corollary}
Fix $k \geq 2$.  Let $H$ be connected and non-regular, and let $G$ be an $n$-vertex $k$-connected graph.  Then there exists a constant $c(k,H)$ so that for $n \geq c(k,H)$ we have 
\[
\hom(G,H) \leq \hom(K_{k,n-k},H),
\]
with equality if and only if $G=K_{k,n-k}$.
\end{corollary}

The proof of Theorem \ref{thm-main} uses a stability technique similar to that in \cite{Engbers} and \cite{EngbersGalvin}.  In particular, we first show that if a connected $n$-vertex graph with minimum degree $\delta$ isn't structurally close to $K_{\delta,n-\delta}$ (specifically, if it doesn't contain a subgraph isomorphic to $K_{\delta,bn}$ for some small constant $b$), then it admits fewer $H$-colorings than $K_{\delta,n-\delta}$.  If it is structurally close to $K_{\delta,n-\delta}$ but isn't $K_{\delta,n-\delta}$, then a second argument will show that it again admits fewer $H$-colorings than $K_{\delta,n-\delta}$. 

It is shown in \cite{EngbersGalvin} that when $H=K_4$, for example, $\hom(C_n,K_4) > \hom(K_{2,n-2},K_4)$, showing that the conclusion of Theorem \ref{thm-main} doesn't hold in general when removing the non-regularity restriction on $H$, even for connected $H$.  We also show here that when $H=K_q$ for large enough $q$, there are connected $n$-vertex graphs with minimum degree $\delta$ with more $K_{q}$-colorings than $K_{\delta,n-\delta}$. To do so, consider first the graph $\frac{n}{\delta+1}K_{\delta+1}$ and in this graph choose one specified vertex in each copy of $K_{\delta+1}$, form a star on those specified vertices, and let $G_1$ be the resulting graph. 

\begin{theorem}\label{thm-cols}
Fix $\delta \geq 2$ 
and let $H=K_{q}$.  
For sufficiently large $n$ (depending on $\delta$, with $(\delta+1)|n$) and sufficiently large $q$ (depending on $n$ and $\delta$) we have $$\hom(K_{\delta,n-\delta},H) < \hom(G_1,H),$$ where $G_1$ is the graph in the preceding paragraph.
\end{theorem}

While the graph $G_1$ in Theorem \ref{thm-cols} has more $K_{q}$-colorings than $K_{\delta,n-\delta}$, it isn't clear which connected $n$-vertex graph with minimum degree $\delta$ maximizes the number of $K_{q}$-colorings in this family.

\section{Proofs of Theorems \ref{thm-main} and \ref{thm-cols}}

Recall that a loop on a vertex adds one to the degree and that $\Delta$ refers to the maximum degree of a vertex in $H$. We begin with a few useful lemmas.

\begin{lemma}\label{lem-longpath}
Fix $\delta$. For non-regular connected $H$ there exists a constant $\ell$ (depending on $\delta$ and $H$) such that if $k \geq \ell$, then $hom(P_k, H) < \Delta^{k-\delta}$.
\end{lemma}

\begin{proof}
The proof of Lemma 3.1 in \cite{EngbersGalvin} shows that for non-regular connected $H$ we have $\hom(P_k,H) \leq |V(H)|^2 c_1 \lambda^k$ for some $\lambda < \Delta$ (where here $c_1$ and $\lambda$ are constants depending on $H$).  This implies the result.
\end{proof}

For shorter paths, we have the following, which is Lemma 4.5 from \cite{EngbersGalvin}.

\begin{lemma}[\cite{EngbersGalvin}]\label{lem-shortpath}
Suppose $H$ is not the complete looped graph on $\Delta$ vertices or $K_{\Delta,\Delta}$. Then for any two vertices $i$, $j$ of $H$ and for $k \geq 4$ there are at most $(\Delta^2-1)\Delta^{k-4}$ $H$-colorings of $P_k$ that map the initial vertex of that path to $i$ and the terminal vertex to $j$.
\end{lemma}

Notice that if $H$ is a non-regular graph, then the conclusion of Lemma \ref{lem-shortpath} applies. Furthermore, we remark that colorings of $P_k$ where the endpoints receive the same color (i.e. $i=j$ in Lemma \ref{lem-shortpath}) are in bijection with colorings of $C_{k-1}$ with one fixed vertex receiving that color. In other words, Lemma \ref{lem-shortpath} may be applied on a cycle $C_{k-1}$ that has one vertex already colored. 

We are now ready to prove Theorem \ref{thm-main}. 

\begin{proof}[Proof of Theorem \ref{thm-main}]
Let $\delta\geq 2$ be fixed and $H$ be a connected non-regular graph.  Following a definition from \cite{Engbers}, we let $S(\delta,H)$ denote the vectors in $V(H)^\delta$ with the property that the elements of the vector have $\Delta$ common neighbors in $H$, and let $s(\delta,H) = |S(\delta,H)|$. Notice that  $S(\delta,H)$ is always non-empty, since if $v$ is any vertex of degree $\Delta$ in $H$ then $(v,v,\ldots,v)$ is in $S(\delta,H)$. Then we have the lower bound 
\begin{equation}\label{eqn-lb}
\hom(K_{\delta,n-\delta},H) \geq s(\delta,H)\Delta^{n-\delta} \geq \Delta^{n-\delta},
\end{equation}
which colors the size $\delta$ partition class of $K_{\delta,n-\delta}$ using an element of $S(\delta,H)$, and independently colors each vertex of the size $n-\delta$ partition class using one of the $\Delta$ common neighbors.

Now let $G$ be a connected $n$-vertex graph with minimum degree $\delta$. The proof will proceed by iteratively coloring parts of the graph $G$; throughout we use connectivity to always color vertices adjacent to a previously colored vertex (after an initial subset of vertices is colored). This iterative coloring procedure produces an upper bound on $\hom(G,H)$ which in most cases is smaller than the lower bound on $\hom(K_{\delta,n-\delta},H)$ given in (\ref{eqn-lb}). 

Suppose that $G$ has a path of length $k \geq \ell$ (with $\ell$ as given in Lemma \ref{lem-longpath}). Coloring the path, and then iteratively coloring the remaining vertices --- which each have at most $\Delta$ possibilities for their color --- we have 
\[
\hom(G,H) < \Delta^{k-\delta}\Delta^{n-k} = \Delta^{n-\delta} \leq \hom(K_{\delta,n-\delta},H).
\]
 Therefore, from now on we assume that the graph $G$ does not have any paths of length $k \geq \ell$. 

Fix a vertex $v$ and color that vertex from $H$; there are $|V(H)|$ ways that this can be done.  Then run the following procedure.
\begin{enumerate}
	\item Denote by $C \subseteq V(G)$ the set of vertices that have already received a color.
	\item Search for a non-trivial component in the induced graph on vertex set $V(G) \setminus C$, i.e., a non-trivial component in the induced graph on the set of uncolored vertices.
	\begin{enumerate}
		\item If there is a non-trivial component that is a tree, find a maximal path $P$ in that component and iteratively color the vertices of the path $P$ (since $\delta(G) \geq 2$, the endpoints of $P$ have a neighbor in $C$).  Then proceed back to step 1.
		\item If all non-trivial components are not trees, choose one such component and fix a cycle in the component.  Let $X$ be the graph consisting of the cycle and the shortest path from a vertex on the cycle to a vertex in $C$.  Iteratively color the vertices of $X$ by coloring the path first and then the cycle.  Then proceed back to step 1.
		\item If all components are trivial, stop.
	\end{enumerate}
\end{enumerate}

Note that since we're assuming that there are no paths of length at least $\ell$ in $G$, at every iteration of this coloring scheme we color fewer than $\ell$ previously uncolored vertices. (In 2(b), removing one edge from the cycle that is adjacent to a vertex on the path makes $X$ into a path.) Also, if we run step 2(a) and color $k_1$ new vertices, Lemma \ref{lem-shortpath} shows that there are at most $(\Delta^2-1)\Delta^{k_1-2}$ ways this can be done. For step 2(b), each vertex in the path of $X$, including one vertex on the cycle $C_{k_2}$, has at most $\Delta$ possible colors.  And once one vertex on the cycle has been colored, the discussion following the statement of Lemma \ref{lem-shortpath} shows that there are at most $(\Delta^2-1)\Delta^{k_2-3}$ ways to color the remaining $k_2-1$ vertices of the cycle. In short, each iteration that colors $k$ previously uncolored vertices can be done in at most $(1-\frac{1}{\Delta^2})\Delta^k$ ways, regardless of the colors appearing on the vertices of $C$.

If there are $a$ iterations of this procedure run, then by coloring any remaining vertices using at most $\Delta$ possible choices for each uncolored vertex, we have
\[
\hom(G,H) \leq |V(H)|(1-\frac{1}{\Delta^2})^a\Delta^n < \Delta^{n-\delta} \leq \hom(K_{\delta,n-\delta},H)
\]
whenever $a > \Delta^2 \left( \log |V(H)| +  \delta \log \Delta \right)$; here $|V(H)|$ comes from the initial vertex, and $(1-\frac{1}{\Delta^2})\Delta^k$ comes from each iteration that colors $k$ new vertices.

So then suppose that $a \leq \Delta^2 \left( \log |V(H)| +  \delta \log \Delta \right)$.  In this case, the final output of colored vertices $C \subseteq V(G)$ upon termination has $|C|$ bounded above by a constant depending on $\delta$ and $H$, and therefore $n-|C|$ vertices remain uncolored. Since the procedure terminated, the vertices in $V(G) \setminus C$ form an independent set, and so in particular each vertex $v \in V(G) \setminus C$ has at least $\delta$ neighbors in the set $C$. By the pigeonhole principle, some set of $\delta$ vertices (among the set $C$) must have at least $(n-|C|)/\binom{|C|}{\delta}$ common neighbors.  In other words, the graphs $G$ for which we do not yet have $\hom(G,H) < \hom(K_{\delta,n-\delta},H)$ are those which have a (not necessarily induced) subgraph isomorphic to $K_{\delta,bn}$ for some small constant $b$ (depending on $\delta$ and $H$). Note that these graphs are ``structurally close'' to $K_{\delta,n-\delta}$.

So, finally, suppose that $G \neq K_{\delta,n-\delta}$ contains $K_{\delta, bn}$ for some constant $b$ (depending on $\delta$ and $H$). We will consider two cases: when $G$ does not have an isomorphic copy of $K_{\delta,n-\delta}$ as a subgraph, and finally when it does.

\medskip

\textbf{$G$ contains no subgraph isomorphic to $K_{\delta,n-\delta}$:} Suppose first that $G$ does not contain a subgraph isomorphic to $K_{\delta,n-\delta}$.  This means that the size $\delta$ partition class in $K_{\delta,bn}$ is not a dominating set.  In particular, the induced subgraph on the vertices outside of the size $\delta$ partition class must have a non-trivial component.  As before, if some component is a tree, let $X$ be a maximal path in the tree; if all components are not trees, let $X$ be the union of a cycle and a shortest path from a vertex on the cycle to a vertex in the size $\delta$ partition class.  
With $X$ now defined, we partition the $H$-colorings of $G$ based on whether the colors in the size $\delta$ partition class form a vector in $S(\delta,H)$ or not.  

If they do form a vector in $S(\delta,H)$, we then color $X$, and then the rest of the graph.  Using Lemma \ref{lem-shortpath} and discussion following it on $X$, this gives at most
\[
s(\delta,H) (\Delta^2-1) \Delta^{n-\delta-2}
\]
$H$-colorings of $G$ of this type.

If the colors on the size $\delta$ partition class do not form a vector in $S(\delta,H)$, then the at least $bn$ neighbors have at most $\Delta-1$ possible choices for their color. Using at most $\Delta$ choices for each of the remaining $n-\delta-bn$ vertices, there are at most
\[
|V(H)|^\delta \cdot (\Delta-1)^{bn} \Delta^{n-\delta-bn}
\]
$H$-colorings of $G$ of this type.

Putting these together, we have
\begin{eqnarray*}
\hom(G,H) &\leq& s(\delta,H) (\Delta^2-1) \Delta^{n-\delta-2} + |V(H)|^\delta \cdot (\Delta-1)^{bn} \Delta^{n-\delta-bn}\\
&\leq& s(\delta,H) \Delta^{n-\delta} - s(\delta,H)\Delta^{n-\delta-2} + |V(H)|^\delta \cdot e^{-bn/\Delta} \Delta^{n-\delta}\\
&<& s(\delta,H) \Delta^{n-\delta}
\end{eqnarray*}
where the final inequality holds for all sufficiently large $n$.

\medskip

\textbf{$G$ contains a subgraph isomorphic to $K_{\delta,n-\delta}$:}
Lastly, suppose that $G$ contains a subgraph isomorphic to $K_{\delta,n-\delta}$.  Since $G \neq K_{\delta,n-\delta}$, it suffices to show that adding any edge to $K_{\delta,n-\delta}$ strictly decreases the number of $H$-colorings. To this end, we suppose that $G$ is obtained from $K_{\delta,n-\delta}$ by the addition of a single edge, and we show that each $H$-coloring of $K_{\delta,n-\delta}$ is also an $H$-coloring of $G$ only when $H$ is the complete looped graph. 

Suppose that $i$ and $j$ are distinct adjacent vertices of $H$.  Then the mapping that sends the size $\delta$ partition class of $K_{\delta,n-\delta}$ to $i$ and the other partition class to $j$ is an $H$-coloring of $K_{\delta,n-\delta}$, and similarly the mapping that sends the size $\delta$ partition class to $j$ and the other partition class to $i$ is another $H$-coloring of $K_{\delta,n-\delta}$.  This is only an $H$-coloring of $G$ if $i$ and $j$ are both looped.  By similar reasoning, if $i$ and $j$ are adjacent in $H$ and $j$ and $k$ are also adjacent in $H$, then $i$ and $k$ must be adjacent.  As $H$ is connected, this implies that $H$ is a fully looped complete graph.  So if $G$ is obtained from $K_{\delta,n-\delta}$ by adding an edge and $H$ is non-regular, then $\hom(G,H) < \hom(K_{\delta,n-\delta},H)$.
\end{proof}

We conclude this section with the proof of Theorem \ref{thm-cols}.

\begin{proof}[Proof of Theorem \ref{thm-cols}]
Let $H=K_{q}$. First, we find an upper bound on $\hom(K_{\delta,n-\delta},K_q)$. There are
\[
q(q-1)\cdots(q-\delta+1)(q-\delta)^{n-\delta}
\]
colorings of $K_{\delta,n-\delta}$ that use distinct colors on the size $\delta$ partition class. In a similar way, there are at most
\[
\delta^2 \cdot q \cdot q^{\delta-2} \cdot (q-1)^{n-\delta} \leq \delta^2 q^{n-1}
\]
colorings of $K_{\delta,n-\delta}$ that have the same color on two or more vertices in the size $\delta$ partition class. This means that there are at most
\[
q(q-1)\cdots(q-\delta+1)(q-\delta)^{n-\delta} + \delta^2 q^{n-1}
\]
$K_{q}$-colorings of $K_{\delta,n-\delta}$.

For the graph $G_1$, we first color the $K_{\delta+1}$ containing the center of the star, and then the remaining copies of $K_{\delta+1}$.  This gives
\begin{eqnarray*}
\hom(G_1,K_{q}) &=& q(q-1)\cdots(q-\delta)\left[ (q-1)(q-1)(q-2)\cdots(q-\delta)\right]^{\frac{n}{\delta+1} - 1}.
\end{eqnarray*}
The coefficient of $q^{n-1}$ in the upper bound on $\hom(K_{\delta,n-\delta},H)$ is $-n\delta + \frac{3\delta^2+\delta}{2}$ and the coefficient of $q^{n-1}$ in $\hom(G_1,K_q)$ is $-\frac{n}{\delta+1}+1-\frac{n\delta}{2}$. 
So $\hom(G_1,K_q) - \hom(K_{\delta,n-\delta},K_q)$ is bounded below by a polynomial in $q$ of degree $n-1$ with leading coefficient $\frac{n\delta}{2}-\frac{n}{\delta+1}+\frac{-3\delta^2-\delta+2}{2}$.
This shows that for $\delta \geq 2$, sufficiently large $n$ (depending on $\delta$), and sufficiently large $q$ (depending on $n$ and $\delta$) we have $\hom(K_{\delta,n-\delta},K_{q}) < \hom(G_1,K_{q})$.
\end{proof}


\section{Concluding Remarks}

Here we highlight a few open questions related to the contents of this paper.  We showed that for sufficiently large $n$ and connected non-regular $H$, the number of $H$-colorings of a connected $n$-vertex graph $G$ with minimum degree $\delta$ is maximized uniquely when $G=K_{\delta,n-\delta}$.  

First, it would be interesting to know if this behavior holds for all $n \geq 2\delta$, as it does for independent sets \cite{CutlerRadcliffe}.
\begin{question}
Fix a non-regular $H$ and $n\geq 2\delta$. Does $K_{\delta,n-\delta}$ maximize the number of $H$-colorings over all connected $n$-vertex graphs with minimum degree $\delta$?
\end{question}
The behavior for non-trivial regular $H$ is still unknown.
\begin{question}
Fix a regular $H$ and $n \geq 2\delta$. Which connected $n$-vertex graph with minimum degree $\delta$ maximizes the number of $H$-colorings?
\end{question}
In particular, we have the following interesting extremal question for proper $q$-colorings.
\begin{question}
For a given $\delta$, $n$, and $q$, which connected $n$-vertex graph with minimum degree $\delta$ maximizes the number of proper $q$-colorings?
\end{question}
It was shown \cite{EngbersGalvin} that in the case $\delta=2$ we have $\hom(K_{2,n-2},K_q)<\hom(C_n,K_q)$ for all fixed $q \geq 4$ and $n \geq 3$, and so in particular $K_{\delta,n-\delta}$ is not always the maximizing graph for fixed $q$, even for sufficiently large $n$.  Note also that when $q=2$, any bipartite graph maximizes the number of $K_{q}$-colorings. 

Investigating what happens when $n<2\delta$ would also be interesting; notice that in this range all $n$-vertex graphs with minimum degree at least $\delta$ are connected. The the following question has been answered in the special case of independent sets \cite{CutlerRadcliffe}.
\begin{question}
Fix any $H$, and let $n<2\delta$.  Which $n$-vertex graph with minimum degree $\delta$ maximizes the number of $H$-colorings?
\end{question}
We can ask these questions for the family of all $n$-vertex graphs with minimum degree $\delta$; answers for some $H$ and sufficiently large $n$ are given in \cite{Engbers}.
\begin{question}
Fix any $H$. Which $n$-vertex graph with minimum degree $\delta$ maximizes the number of $H$-colorings?
\end{question}
\begin{question}
For a given $\delta$, $n$, and $q$, which $n$-vertex graph with minimum degree $\delta$ maximizes the number of proper $q$-colorings?
\end{question}
Note that when $q=2$, we seek the $n$-vertex graph with minimum degree $\delta$ that has the largest number of bipartite components, and so the maximizing graph is $\frac{n}{2\delta}K_{\delta,\delta}$.



\end{document}